\documentclass[10pt]{article}
\usepackage{amssymb,amsmath,graphicx,amsthm,mathrsfs}
\newtheorem{Theorem}{Theorem}[section]
\newtheorem{Lemma}[Theorem]{Lemma}

\newtheorem{Definition}[Theorem]{Definition}

\newtheorem{Remark}[Theorem]{Remark}
\numberwithin{equation}{section}

\title{A Global Uniqueness for Formally Determined Inverse Electromagnetic Obstacle Scattering}

\author{Hongyu Liu\thanks{Department of Mathematics,
University of Washington, Box 354350, Seattle, WA 98195, USA ({\tt
hyliu@math.washington.edu}).}}

\begin{document}

\date{}

\maketitle

\begin{abstract}
It is proved that a general polyhedral perfect conducting obstacle
in $\mathbb{R}^3$, possibly consisting of finitely many solid
polyhedra, is uniquely determined by the far-field pattern
corresponding to a single incident wave. This improves earlier
results in the literature to the formally determined case.\\


\end{abstract}

\section{Introduction}

In this paper, we shall be mainly concerned with the inverse
electromagnetic obstacle scattering, where one utilizes the
time-harmonic electromagnetic far-field measurements to identify the
inaccessible unknown impenetrable objects.

For a brief description of the forward scattering problem, we let a
perfect conducting obstacle $\mathbf{D}\subset \mathbb{R}^3$ be a
compact set with connected Lipschitz complement
$\mathbf{G}:=\mathbb{R}^3\backslash \mathbf{D}$, and
\begin{align}
\mathbf{E}^i(x):=&\frac{\mathrm{i}}{k}\mbox{curl}\ \mbox{curl}\,p\,
e^{\mathrm{i}k {x}\cdot d}=\mathrm{i}k(d\times p)\times d
e^{\mathrm{i}k {x}\cdot d},\label{eq:electric}\\
\mathbf{H}^i(x):=& \mbox{curl}\,p\,e^{\mathrm{i}k {x}\cdot
d}=\mathrm{i}k d\times p e^{\mathrm{i}k {x}\cdot
d},\label{eq:magnetic}
\end{align}
be the incident electric and magnetic fields, where $p\in
\mathbb{R}^3$, $k>0$ and $d\in \mathbb{S}^2:=\{{x}\in\mathbb{R}^3;
|{x}|=1\}$ represents respectively the polarization, wave number and
direction of propagation. The incident wave propagating in the
homogeneous background medium will be perturbed when it encounters
an obstacle, and produces a scattered filed. We denote by
$\mathbf{E^s}$ and $\mathbf{H^s}$ the scattered electric and
magnetic fields respectively, and define the the total electric and
magnetic fields to be
\begin{equation}\label{eq:total field}
 \mathbf{E}({x})=\mathbf{E^i}({x})+\mathbf{E^s}({x}),\quad \mathbf{H}(x)=\mathbf{H}^i(x)+ \mathbf{H}^s(x)\qquad {x}\in\mathbb{R}^3.
\end{equation}
Then the direct scattering problem consists of finding a solution
$(\mathbf{E},\mathbf{H})\in H_{loc}^1(\mbox{curl}; \mathbf{G})$
$\times H_{loc}^1(\mbox{curl}; \mathbf{G})$ that satisfies the
following time-harmonic Maxwell equations
\begin{align}
&\qquad\mbox{curl}\,\mathbf{E}-\mathrm{i}\,k\,\mathbf{H}=0,\quad
\mbox{curl}\, \mathbf{H}+\mathrm{i}\,k\,\mathbf{E}=0\quad \mbox{in ~
$\mathbf{G}:=\mathbb{R}^3\backslash
\mathbf{D}$},\label{eq:Maxwell}\\
& \hspace*{2cm}~ \nu\times \mathbf{E}=0\quad \mbox{on ~$\partial
\mathbf{G}$},\label{eq:perfect bc}\\
&\qquad \qquad \lim_{|{x}|\rightarrow\infty}(\mathbf{H^s}\times {x}-
|{x}| \mathbf{E^s})=0,\label{eq:silver}
\end{align}
where the last limit corresponds to the so-called Silver-M\"uller
radiation condition characterizing the fact that the scattered wave
is radiating.

The well-posedness of the forward scattering problem
(\ref{eq:electric})-(\ref{eq:silver}) has been well understood (see
\cite{Ces}). Particularly, the cartesian components of $\mathbf{E}$
and $\mathbf{H}$ are (real) analytic in $\mathbf{G}$ and the
asymptotic behavior of the scattered field ($\mathbf{E^s}$,
$\mathbf{H^s}$) is governed by (see \cite{ColKre})

\begin{align}
\mathbf{E^s}({x};\mathbf{D},p,k,d)=&\frac{e^{\mathrm{i}k{x}\cdot
d}}{|{x}|}\bigg\{\mathbf{E}_\infty(\hat{{x}};\mathbf{D},p,k,d)+\mathcal{O}(\frac{1}{|{x}|})\bigg\}
\quad \mbox{as} ~~|{x}|\rightarrow\infty, \label{far1}\\
\mathbf{H^s}({x};\mathbf{D},p,k,d)=&\frac{e^{\mathrm{i}k{x}\cdot
d}}{|{x}|}\bigg\{\mathbf{H}_\infty(\hat{{x}};\mathbf{D},p,k,d)+\mathcal{O}(\frac{1}{|{x}|})\bigg\}
\quad \mbox{as} ~~|{x}|\rightarrow\infty,\label{far2}
\end{align}
uniformly for all $\hat{{x}}={x/|x|}\in \mathbb{S}^2$. The functions
$\mathbf{E}_\infty(\hat{{x}})$ and $\mathbf{H}_\infty(\hat{{x}})$ in
(\ref{far1}) and (\ref{far2}) are called, respectively, the electric
and magnetic \emph{far-field patterns}, and both are analytic on the
unit sphere ~$\mathbb{S}^2$. As is noted above,
$\mathbf{E^s}({x};\mathbf{D},p,k,d)$,
$\mathbf{E}_\infty(\hat{{x}};\mathbf{D},p,k,d)$, etc. will be
frequently used to specify their dependence on the observation
direction $\hat{x}$, the polarization $p$, the wave number $k$ and
the incident direction $d$.

Now, the inverse scattering problem is the following. Assume that
the obstacle $\mathbf{D}$ is unknown or inaccessible and we aim to
image the object and thereby identify it by performing far-field
measurements. That is, with the measurement of the electric
far-field pattern (or, equivalently, the magnetic far-field pattern)
of the wave which is scattered by $\mathbf{D}$ corresponding to a
given incident wave, for one or more choices of its polarization $p$
or of its wave number $k$, or of its propagation direction $d$, we
would like to recover the obstacle whose scattered waves are
compatible with the measurements performed. From the mathematical
viewpoint, the inverse obstacle scattering can be formulated as the
following operator equation
\begin{equation}\label{eq:IEOSP}
\mathcal{F}_e (\partial \mathbf{G})= \mathbf{E}_\infty (\hat{{{x}}};
\mathbf{D},p,k,d)\qquad\mbox{for $(\hat{{x}},p,k,d)\in
\mathbb{S}^2_0\times\mathbb{U}\times \mathbb{K}\times
\widetilde{\mathbb{S}}^2_0$},
\end{equation}
where $\mathbb{S}_0^2, \widetilde{\mathbb{S}}_0^2\subset
\mathbb{S}^2$, $\mathbb{U}\subset \mathbb{R}^3$ ,$\mathbb{K}\subset
\mathbb{R}_{+}:=\{x\in\mathbb{R}; x>0\}$ and the nonlinear operator
$\mathcal{F}_e$ is defined by the forward scattering system
(\ref{eq:Maxwell})-(\ref{eq:silver}). The inverse obstacle
scattering, having its roots in the technology of radar and sonar,
are also central to many other areas of science such as medical
imaging, geophysical exploration and nondestructive testing, etc..
We refer to \cite{ColKre} for a more detailed discussion and related
literature. As usual in most of the inverse problems, the first
question to ask in this context is the {\it identifiability}; i.e.,
whether an obstacle can really be identified from a knowledge of its
far-field pattern.  Mathematically, the \emph{identifiability} is
the \emph{uniqueness} issue, which is the injectivity of the
(nonlinear) operator $\mathcal{F}_e$ in (\ref{eq:IEOSP}). That is,
\begin{quote}
\emph{If two obstacles $\mathbf{D}$ and $\widetilde{\mathbf{D}}$
produce the same far field data, i.e.,
$$
\mathbf{E}_\infty(\hat{{x}};\mathbf{D},p,k,d)
=\mathbf{E}_\infty(\hat{{x}};\widetilde{\mathbf{D}},p,k,d) \quad
\mbox{for $(\hat{{x}},p,k,d)\in \mathbb{S}^2_0\times\mathbb{U}\times
\mathbb{K}\times \widetilde{\mathbb{S}}^2_0$},
$$
does $\mathbf{D}$ have to be the same as ~$\widetilde{\mathbf{D}}$ ?
}
\end{quote}
We refer to \cite{Isa} for a general discussion of the critical role
of uniqueness which plays in inverse problems theory theoretically
as well as numerically. It is observed that the uniqueness results
also provide the practical information on how many measurement data
one should use to identify the underlying object. As an important
ingredient in the uniqueness study and noting $\mathbf{E}_\infty$ is
an analytic function, one sees that if $\mathbb{S}_0^2$ in
(\ref{eq:IEOSP}) is an open subset of the unit sphere, no matter how
small the subset is, we can always recover such data on the whole
unit sphere by analytic continuation. Hence, for our uniqueness
study, without loss of generality, we can assume that the far-field
data are given on the whole unit sphere, i.e., in every possible
observation direction. Then it is easily seen that the inverse
obstacle scattering is formally determined with fixed $p_0\in
\mathbb{R}^3$, $k_0>0$ and $d_0\in \mathbb{S}^2$, since the far
field data depend on the same number of variables, as does the
obstacle which is to be recovered.\footnote{Here, the number of
variables is 2, since both $\partial \mathbf{G}$ and $\mathbb{S}^2$
are 2-manifold.} Due to such observation, there is a widespread
belief that one can establish the uniqueness by using the far field
pattern corresponding to a single incident wave. However, this has
remained to be a longstanding challenging open problem, though
extensive study has been made in this aspect (see \cite{CakCol2} and
\cite{ColKre3}). The only previous result that we are aware of this
kind is in \cite{Kre}, where it is shown that a simple ball can be
uniquely determined by its far-field measurement corresponding to a
single incident wave.

In the past few years, significant progress has been achieved on the
unique determination of general polyhedral type obstacles by several
far-field measurements. The breakthrough is first made in the
inverse acoustic obstacle scattering, where one utilizes the
acoustic far-field measurement to identify the underlying scattering
objects (see \cite{AleRon} \cite{CheYam} \cite{ElsYam2}
\cite{LiuZou} \cite{LiuZou4}). Among the arguments for the proofs of
those results, the new methodology developed in \cite{LiuZou} which
we call \emph{path argument} is proved to be particularly suitable
for attacking such problems. Based on suitably devised \emph{path
arguments}, together with some novel reflection principles for the
solutions of Maxwell equations, various uniqueness results have been
established in different settings with general polyhedral type
obstacles in \cite{LiuYamZou1} and \cite{LiuYamZou2}, but all with
the far-field measurements corresponding to two different incident
waves. In the current work, we are able to improve significantly on
this result to the formally determined setting. It is shown that the
measurement of the far-field pattern corresponding to a single
incident wave uniquely determines a general polyhedral perfect
conducting obstacle. For the proof, we follow the general strategy
in \cite{LiuYamZou1} and \cite{LiuYamZou2}, but several technical
new ingredients must be developed and the \emph{path argument} in
this work is refined significantly. We next state more precisely the
main result.

It is first recalled that a compact polyhedron in $\mathbb{R}^3$ is
a simply connected compact set whose boundary is composed of (open)
\emph{faces, edges} and \emph{vertices}. In the sequel, we call
$\mathbf{D}$ a \emph{polyhedral obstacle} if it is composed of
finitely many (but unknown \emph{a priori}) pairwise disjoint
compact polyhedra. That is,
\begin{equation}\label{eq:polyhedral obstacle}
\mathbf{D}=\bigcup_{l=1}^{m} D_l,
\end{equation}
where $m$ is an unknown integer but must be finite and each $D_l,
1\leq l\leq m$ is a compact polyhedron such that
\[
D_j\cap D_{j'}=\emptyset\quad \mbox{if $j\leq j'$ and $1\neq j,
j'\leq m$}.
\]
Clearly, the forward scattering problem
(\ref{eq:Maxwell})-(\ref{eq:silver}) with such a polyhedral obstacle
$\mathbf{D}$ is well-posed. Moreover, we know that the singular
behaviors of the weak solution only attach to the edges and
vertices, that is, $(\mathbf{E}, \mathbf{H})$ satisfies
(\ref{eq:Maxwell}) in the classical sense in any subdomain of
$\mathbf{G}$, which does not meet any corner or edge of $\mathbf{D}$
(see \cite{CosDau}). By the regularity of the strong solution for
the forward scattering problem, we know that both $\mathbf{E}$ and
$\mathbf{H}$ are at least $C^{0,\alpha}$-continuous ($0<\alpha<1$)
up to the regular points, namely, points lying in the interior of
the open faces of $\mathbf{D}$.

The main result of this paper is the following:

\begin{Theorem}\label{thm:single far-field}
Let $\mathbf{D}$ and $\widetilde{\mathbf{D}}$ be two perfect
polyhedral obstacles. For any fixed $k_0>0$, $d_0\in \mathbb{S}^2$
and $p_0\in \mathbb{R}^3$ such that $d_0$ and $p_0$ are linearly
independent, we have $\mathbf{D}=\mathbf{\widetilde{D}}$ as long as
\begin{equation}
\mathbf{E}_\infty({\hat{x}}; \mathbf{D}, p_0, k_0,
d_0)=\mathbf{E}_\infty({\hat{x}}; \mathbf{\widetilde{D}}, p_0, k_0,
d_0)\qquad \mbox{for ${\hat{x}}\in \mathbb{S}^2$}.
\end{equation}

\end{Theorem}

\begin{Remark}
As mentioned earlier, there are some uniqueness results established
in \cite{LiuYamZou1} and \cite{LiuYamZou2} in the unique
determination of general polyhedral obstacles, but with the
far-field data corresponding to two different incident waves.
However, the polyhedral obstacles considered in \cite{LiuYamZou1}
are more general than the present ones, and they admit the
simultaneous presence of crack-type components (namely, screens).
Whereas the uniqueness in \cite{LiuYamZou2} is established without
knowing the \emph{a priori} physical properties of the underlying
obstacle. In Section~4, we would make concluding remarks on that the
uniqueness result in Theorem~\ref{thm:single far-field} can not
cover completely the ones obtained in \cite{LiuYamZou1} and
\cite{LiuYamZou2}.
\end{Remark}

%
%
%

The rest of the paper is organized as follows. In Section~2, we
introduce the perfect set and perfect planes, and then show several
crucial properties of them which shall play a key role in proving
Theorem~\ref{thm:single far-field}. Section~3 is devoted to the
proof of Theorem~\ref{thm:single far-field}, and in Section~4, we
give some concluding remarks.

\section{Perfect Set and Perfect Planes}\label{sect:key lemmas}

First, we fix some notations which shall be used throughout of the
rest of the paper. We denote an open ball in $\mathbb{R}^3$ with
center ${x}$ and radius $r$ by $B_r({x})$, the closure of $B_r({x})$
by $\bar{B}_r({x})$ and the boundary of $B_r({x})$ by $S_r({x})$.
The notation $T_r({x})$ is defined to be an open cube of edge length
$r$, centered at ${x}$, while $\bar{T}_r({x})$ is its corresponding
closure. Unless specified otherwise, $\nu$ shall always denote the
inward normal to a concerned domain, or the normal to an
two-dimensional plane in $\mathbb{R}^{3}$. The distance between two
sets $\mathscr{A}$ and $\mathscr{B}$ in $\mathbb{R}^3$ is understood
as usual to be $\mathbf{d}(\mathscr{A}, \mathscr{B})=\inf_{{x}\in
\mathscr{A}, {y}\in \mathscr{B}}|{x}-{y}|$. Finally, a curve
$\gamma=\gamma(t) (t\geq 0)$ is said to be regular if it is
$C^1$-smooth and $\frac{d}{dt}\gamma(t)\neq 0$.

Henceforth, we let $k_0>0$, $d_0\in \mathbb{S}^2$ and $p_0\in
\mathbb{R}^3$ be fixed such that $d_0$ and $p_0$ are linearly
independent, and denote by $\mathbf{E(x)}:=\mathbf{E}({x};
\mathbf{D}, p_0, k_0, d_0)$ the total electric field in
(\ref{eq:Maxwell})-(\ref{eq:silver}) corresponding to a polyhedral
perfect conducting obstacle $\mathbf{D}$ as described in
(\ref{eq:polyhedral obstacle}). The following definition of a
\emph{perfect set} is modified from that in \cite{LiuYamZou1} to fit
the problem being under investigation.

\begin{Definition}\label{def:perfect set}
$\mathscr{P}_\mathbf{E}$ is called a perfect set of $\mathbf{E}$ in
$\mathbf{G}:=\mathbb{R}^3\backslash \mathbf{D}$ if
\begin{equation*}
\mathscr{P}_{\mathbf{E}}=\left\{{x}\in \mathbf{G}; \nu\times
\mathbf{E} \mid_{\Pi\cap B_r({x})\cap \mathbf{G}}=0\,\, \mbox{for
some $r>0$ and plane $\Pi$ passing through ${x}$}\right\},
\end{equation*}
where $\nu$ is the unit normal to the plane $\Pi$.
\end{Definition}

For any ${x}\in \mathscr{P}_\mathbf{E}$, we let $\Pi$ be the plane
involved in the definition of $\mathscr{P}_{\mathbf{E}}$.
Furthermore, we let $\widetilde{\Pi}$ be the connected component of
$\Pi\backslash \mathbf{D}$ containing ${x}$, then by the analyticity
of $\mathbf{E}$ in $\mathbf{G}$, we see $\nu\times \mathbf{E}=0$ on
${\widetilde{\Pi}}$ by classical continuation. In the sequel, such
$\widetilde{\Pi}$ will be referred to as a \emph{perfect plane}. The
introduction of the prefect set and perfect plane is motivated by
the observation that, when proving Theorem~\ref{thm:single
far-field} by contradiction, if two different obstacles produce the
same far-field pattern, then outside one obstacle there exists a
perfect plane which is extended from an open face of the other
obstacle. Starting from now on, $\widetilde{\Pi}_l$ with an integer
$l$, shall always represent a perfect plane in $\mathbf{G}$ which
lies on the plane $\Pi_l$ in $\mathbb{R}^{3}$.

A very fine property of perfect planes is the so-called
\emph{reflection principle}, which constitutes an indispensable
ingredient in the path arguments for proving the uniqueness results
in \cite{LiuYamZou1} and \cite{LiuYamZou2}. We formulate the
principle in the following theorem. Subsequently, we use
$\mathscr{R}_\Pi$ to denote the reflection in $\mathbb{R}^3$ with
respect to a plane $\Pi$.

\begin{Theorem}\label{thm:reflection}
For a connected polyhedral domain $\Omega$ in
$\mathbf{G}:=\mathbb{R}^3\backslash \mathbf{D}$, let
$\widetilde{\Pi}$ be one of its faces that lies on some perfect
plane. Furthermore, let $\Pi$ be the plane in $\mathbb{R}^3$
containing $\widetilde{\Pi}$ and $\Omega\cup
\mathscr{R}_\Pi\Omega\subset \mathbf{G}$. We have two consequences:

\begin{enumerate}

\item[(i)] $\nu_\Pi\times \mathbf{E}=0$\,\, \mbox{on\, $\Pi\cap(\Omega\cup \mathscr{R}_\Pi
\Omega)$};

\item[(ii)] Suppose that $\Sigma\subset \partial\Omega$ is a subset
of one face of $\Omega$ other than $\widetilde{\Pi}$, and the
following condition holds
\begin{equation}\label{eq:thm perfect new1}
\nu_\Sigma\times \mathbf{E}=0\qquad \mbox{on $\Sigma$},
\end{equation}
where $\nu_\Sigma$ is the unit normal to $\Sigma$ directed to the
interior of $\Omega$. Then we have
\begin{equation}\label{eq:thm perfect new2}
\nu_{\Sigma'}\times \mathbf{E}=0\qquad \mbox{on $\Sigma'$},
\end{equation}
where $\Sigma'=\mathscr{R}_\Pi\Sigma$ and $\nu_{\Sigma'}$ is the
unit normal to $\Sigma'$ directed to the interior of
$\mathscr{R}_\Pi\Omega$.

\end{enumerate}

\end{Theorem}

\begin{proof}
The verification for (i) can be found in the proof of Theorem~3.2 in
\cite{LiuYamZou2}, while for (ii), is given in Theorem~2.3 in
\cite{LiuYamZou1}.
\end{proof}

The reflection principle in item (i) of Theorem~\ref{thm:reflection}
is particularly useful when $(\Pi\cap(\Omega\cup \mathscr{R}_\Pi
\Omega))\backslash \widetilde{\Pi}\neq\emptyset$. Clearly, in such
case, we can find a perfect plane also lying on the plane $\Pi$, but
different from $\widetilde{\Pi}$.

Next, we would classify all those perfect planes in $\mathbf{G}$
into two sets in $\mathbb{R}^3$, one is bounded and the other is
unbounded. In fact, it is verified directly that there might exist
unbounded perfect planes\footnote{This constitutes one of the major
differences from those perfect planes introduced in
\cite{LiuYamZou1} and \cite{LiuYamZou2}. All the perfect planes
defined there are bounded due to the use of two different incident
waves. See Lemma~3.2 in \cite{LiuYamZou1}.}. In our subsequent path
argument for proving Theorem~\ref{thm:single far-field}, the
procedure of continuation of perfect planes along an exit path might
be broken down with the presence of an unbounded perfect plane,
since one may not be able to find another perfect plane with an
unbounded perfect plane by using the reflection principle in
Theorem~\ref{thm:reflection}.
%
In the rest of this section, we shall show some critical properties
on the unbounded perfect planes.

\begin{Lemma}\label{lem:conplane}
All the unbounded perfect planes associated with $\mathbf{E}$ in
$\mathbf{G}$ are conplane.
\end{Lemma}
Obviously, Lemma~\ref{lem:conplane} is divided into the following
two lemmata:

\begin{Lemma}\label{lem:nparallel}
There cannot exist two unbounded perfect planes $\widetilde{\Pi}_1$
and $\widetilde{\Pi}_2$ such that $\Pi_1\nparallel \Pi_2$.
\end{Lemma}

\begin{Lemma}\label{lem:parallel}
There cannot exist two different unbounded perfect planes
$\widetilde{\Pi}_1$ and $\widetilde{\Pi}_2$ such that
$\Pi_1\parallel \Pi_2$.
\end{Lemma}

\begin{proof} [Proof of Lemma~\ref{lem:nparallel}]
Assume contrarily that $\widetilde{\Pi}_1$ and $\widetilde{\Pi}_2$
are two unbounded perfect planes in $\mathbf{G}$ such that
$\Pi_1\nparallel \Pi_2$. Let $\nu_1$ and $\nu_2$, respectively, be
the unit normals to $\Pi_1$ and $\Pi_2$. Noting that
$\mathbf{E}^s({x})=\mathcal{O}(1/|{x}|)$ as $|{x}|\rightarrow
\infty$, we have from
\[
\nu_l\times \mathbf{E}(x)=0\quad \mbox{on \ $\widetilde{\Pi}_l$\ for
\ $l=1,2,$}
\]
that
\[
\lim_{{x}\in \widetilde{\Pi}_l: |{x}|\rightarrow \infty}|\nu_l\times
\mathbf{E^i}({x})|=0\quad \mbox{for \ $l=1,2.$}
\]
Using (\ref{eq:electric}), we further deduce
\[
\nu_l\times((d_0\times p_0)\times d_0)=0\quad \mbox{for \ $l=1,2.$}
\]
That is, $\nu_1\parallel \nu_2$ since they are both parallel to a
fixed vector $(d_0\times p_0)\times d_0$, contradicting to our
assumption that $\Pi_1\nparallel \Pi_2$ and completing the proof.
\end{proof}
%
%
%
%
%

\begin{proof}[Proof of Lemma~\ref{lem:parallel}]
By contradiction, we assume that there exist two different perfect
planes $\widetilde{\Pi}_1$ and $\widetilde{\Pi}_2$ such that
$\Pi_1\parallel \Pi_2$. Let $\mathscr{T}:=T_r(0)$ be a sufficiently
large cube such that $\mathbf{D}\subset \mathscr{T}$, and by
suitable rotation, we may without loss of generality assume that
both $\Pi_1$ and $\Pi_2$ are perpendicular to one face of
$\mathscr{T}$. Next, with a little bit abuse of notations, we still
denote by $\widetilde{\Pi}_1$ and $\widetilde{\Pi}_2$ those parts of
$\widetilde{\Pi}_1$ and $\widetilde{\Pi}_2$ lying outside of
$\mathscr{T}$, namely, $\widetilde{\Pi}_1\backslash \mathscr{T}$ and
$\widetilde{\Pi}_2\backslash \mathscr{T}$, and the same rule applies
to $\widetilde{\Pi}_l, l\in \mathbb{Z}$ appearing in the rest of the
proof. Now, in the (unbounded) polyhedral domain
$\mathbb{R}^3\backslash \mathscr{T}$, we can make use of the
reflection reflection as stated in (ii) of
Theorem~\ref{thm:reflection}, and from $\widetilde{\Pi}_1$ and
$\widetilde{\Pi}_2$ to find that
\[
\nu\times \mathbf{E}=0\quad \mbox{on \ \
$\widetilde{\Pi}_3:=\mathscr{R}_{\Pi_2} (\widetilde{\Pi}_1)$}.
\]
Continuing with such argument, from $\widetilde{\Pi}_2$ and
$\widetilde{\Pi}_3$ we have
\[
\nu\times \mathbf{E}=0\quad \mbox{on \ \
$\widetilde{\Pi}_4:=\mathscr{R}_{\Pi_3} (\widetilde{\Pi}_2)$}.
\]
By repeating this reflection, we eventually find a sequence of
perfect planes $\widetilde{\Pi}_l,\ l=1,2,3,\ldots$ such that all
$\widetilde{\Pi}_l$'s are parallel to each other. Clearly,
$\mathbf{d}(\widetilde{\Pi}_l,
\widetilde{\Pi}_{l+1})=\mathbf{d}(\widetilde{\Pi}_1,\widetilde{\Pi}_2)>0$
being fixed for $l=1,2,3,\ldots$. Hence, there must exist some
$l_0<\infty$ such that $\mathscr{T}$ lies entirely at one side of
$\widetilde{\Pi}_{l_0}$. That is, $\widetilde{\Pi}_{l_0}=\Pi_{l_0}$
is the whole plane in $\mathbb{R}^3$. Obviously, $\mathbf{D}$ also
lies at one side of $\Pi_{l_0}$. Using again the reflection
principle in Theorem~\ref{thm:reflection}, (ii), we see
$\nu\times\mathbf{E}=0$ on $\mathscr{R}_{\Pi_{l_0}}(\partial
\mathbf{D})$. Finally, let $\Sigma_1$ and $\Sigma_2$ be two adjacent
faces of $\mathscr{R}_{\Pi_{l_0}}(\partial\mathbf{D})$ and we have
from the extension of $\Sigma_1$ and $\Sigma_2$ two non-parallel
unbounded perfect planes, which contradicts to
Lemma~\ref{lem:nparallel}. The proof is completed.
\end{proof}

We proceed to make an important observation of the reflection
principle (i) in Theorem~\ref{thm:reflection}, when $\Omega\cup
\mathscr{R}_\Pi \Omega$ is unbounded while $\widetilde{\Pi}$ is
bounded. In this case, it is clear that the extension of some part
of $(\Pi\cap(\Omega\cup \mathscr{R}_\Pi \Omega))\backslash
\widetilde{\Pi}$ gives at least one unbounded perfect plane. That
is, some bounded perfect plane might imply the existence of some
correspondingly unbounded perfect plane. Next, we study carefully
such special bounded perfect plane $\widetilde{\Pi}_0$, which can be
regarded as ``unbounded". To localize our investigation, we fix an
arbitrary point $x_0\in \widetilde{\Pi}_0\cap \mathbf{G}$ and take a
sufficiently small ball $\mathrm{B}_0:=B_r(x_0)$ such that
$\mathrm{B}_0\subset \mathbf{G}$.  $\mathrm{B}_0$ is divided by
$\widetilde{\Pi}_0$ into two half balls, which we respectively
denote by $\mathrm{B}_0^+$ and $\mathrm{B}_0^-$. Let
$\mathbf{G}_0^+$ be the connected component of $\mathbf{G}\backslash
\widetilde{\Pi}_0$ containing $\mathrm{B}_0^+$ and $\mathbf{G}_0^-$
be the connected component of $\mathbf{G}\backslash
\widetilde{\Pi}_0$ containing $\mathrm{B}_0^-$. We remark that it
may happen that $\mathbf{G}_0^+=\mathbf{G}_0^-$. Next, let
$\mathbf{\Lambda}_0^+$ be the connected component of
$\mathbf{G}_0^+\cap \mathscr{R}_{\Pi_0}(\mathbf{G}_0^-)$ containing
$\mathrm{B}_0^+$, and $\mathbf{\Lambda}_0^-$ be the connected
component of $\mathbf{G}_0^-\cap
\mathscr{R}_{\Pi_0}(\mathbf{G}_0^+)$ containing $\mathrm{B}_0^-$.
Finally, set $\mathbf{\Lambda}_0=\mathbf{\Lambda}_0^+\cup
\mathbf{\Lambda}_0^-$ and we see that $\mathbf{\Lambda}_0$ is a
polyhedral domain which symmetric with respect to $\Pi_0$, and
moreover, $\mathrm{B}_0\subset \mathbf{\Lambda}_0$. One can easily
see that the construction of $\mathbf{\Lambda}_0$ is only dependent
on the prefect plane $\widetilde{\Pi}_0$. Since $\partial
\mathbf{\Lambda}_0$ is composed of subsets lying either on $\partial
\mathbf{D}$ or on $\mathscr{R}_{\Pi_0}(\partial \mathbf{D})$, by the
reflection principle in (ii) of Theorem~\ref{thm:reflection}, we
have $\nu\times \mathbf{E}=0$ on $\partial \mathbf{\Lambda}_0$.

Starting from now on, we shall denote by
$\mathbf{\Lambda}_{\widetilde{\Pi}_l}$ the symmetric set constructed
as above corresponding to a bounded perfect plane
$\widetilde{\Pi}_l$; namely, in the above,
$\mathbf{\Lambda}_{\widetilde{\Pi}_0}:=\mathbf{\Lambda}_0$. Clearly,
in case $\mathbf{\Lambda}_{\widetilde{\Pi}_l}$ is unbounded, we see
from our earlier discussion that there must exist an unbounded
perfect plane which is extended from some part of $(\Pi_l\cap
\mathbf{\Lambda}_{\widetilde{\Pi}_l})\backslash \widetilde{\Pi}_l$.
%
%
%
%
%
%
%
Such observation in combination with the result in
Lemma~\ref{lem:conplane} gives

\begin{Lemma}\label{lem:b conplane ub}
All the bounded perfect planes $\widetilde{\Pi}_l$ with unbounded
$\mathbf{\Lambda}_{\widetilde{\Pi}_l}$ and all the unbounded perfect
planes are conplane.
\end{Lemma}

Based on Lemma~\ref{lem:b conplane ub}, we introduce the following
set consisting of all the ``unbounded" perfect planes
\begin{align}
&\mathcal{Q}_{\mathbf{E}}:=\{\widetilde{\Pi};\ \widetilde{\Pi}\
\mbox{is an unbounded
perfect plane}\nonumber\\
&\hspace*{2cm}\mbox{or $\widetilde{\Pi}$ is a bounded perfect plane
but with unbounded $\mathbf{\Lambda}_{\widetilde{\Pi}}$}
\}.\label{eq:unbouded perfect planes}
\end{align}
Since all the members in $\mathcal{Q}_\mathbf{E}$ are conplane, one
verifies directly that $\mathcal{Q}_\mathbf{E}$ consists of at most
finitely many perfect planes by noting the fact that $\mathbf{D}$ is
composed of finitely many pairwise disjoint compact polyhedra. We
further define $\breve{\mathcal{Q}}_{\mathbf{E}}$ to be the subset
of $\mathcal{Q}_{\mathbf{E}}$ consisting of those bounded perfect
planes in $\mathcal{Q}_{\mathbf{E}}$. Next, we show some topological
properties of the sets $\mathcal{Q}_\mathbf{E}$ and
$\breve{\mathcal{Q}}_{\mathbf{E}}$.

\begin{Lemma}\label{lem:topology}
Let $\mathbf{G}:=\mathbb{R}^3\backslash \mathbf{D}$, then
\begin{enumerate}
\item[(i)]~$\mathbf{G}\backslash \overline{\breve{\mathcal{Q}}}_{\mathbf{E}}$ is
connected;

\item[(ii)]~$\mathbf{G}\backslash \overline{\mathcal{Q}}_{\mathbf{E}}$ has no
bounded connected component.
\end{enumerate}
\end{Lemma}

\begin{proof}
We first observe that $\breve{\mathcal{Q}}_{\mathbf{E}}$ is bounded
since $\breve{\mathcal{Q}}_{\mathbf{E}} \subset
\overline{ch(\mathbf{D})}$, where $ch(\mathbf{D})$ is the convex
hull of $\mathbf{D}$. By further noting that $\partial \mathbf{G}$
is bounded, we know that $\mathbf{G}\backslash
\overline{\breve{\mathcal{Q}}}_{\mathbf{E}}$ has exactly one
unbounded connected component. Hence, if $\mathbf{G}\backslash
\overline{\breve{\mathcal{Q}}}_{\mathbf{E}}$ is not connected, it
must have some bounded connected component, say
$\mathcal{C}_0\subset \mathbf{G}$. Clearly, there must be one face
of the polyhedral domain $\mathcal{C}_0$ that comes from exactly a
perfect plane in $\breve{\mathcal{Q}}_{\mathbf{E}}$, say
$\widetilde{\Pi}_0$.  Now, one can verify directly that
$\mathbf{\Lambda}_{\widetilde{\Pi}_0}\subset \mathcal{C}_0\cup
\mathscr{R}_{\Pi_0}\mathcal{C}_0$, which is bounded since
$\mathcal{C}_0$ is bounded. But this contradicts to the assumption
that $\widetilde{\Pi}_0\in \mathcal{Q}_{\mathbf{E}}$, thus proving
assertion (i). Next, assertion (ii) is readily seen from (i). In
fact, if $\mathbf{G}\backslash \overline{\mathcal{Q}}_{\mathbf{E}}$
has a bounded connected component, say $\mathcal{D}_0$, then one
must have $\mathcal{D}_0\subset \mathbf{G}\backslash
\overline{\breve{\mathcal{Q}}}_{\mathbf{E}}$, which is certainly not
true. The proof is completed.
\end{proof}

Correspondingly, we set
\begin{equation}
\mathcal{S}_\mathbf{E}=\{\widetilde{\Pi};\ \widetilde{\Pi} \
\mbox{is a bounded perfect plane with bounded
$\mathbf{\Lambda}_{\widetilde{\Pi}}$}\}.
\end{equation}

Finally, we give a lemma concerning the fundamental property of a
connected set (see e.g., Theorem~3.19.9 in \cite{Die}), which shall
be needed in the next section on proving Theorem~\ref{thm:single
far-field}.

\begin{Lemma}\label{lem:fundamental property}
Let $\mathbb{E}$ be a metric space, $\mathscr{A}\subset \mathbb{E}$
be a subset and $\mathscr{B}\subset \mathbb{E}$ be a connected set
such that $\mathscr{A}\cap \mathscr{B}\neq \emptyset$ and
$(\mathbb{E}\backslash \mathscr{A})\cap \mathscr{B}\neq \emptyset$,
then $\partial \mathscr{A}\cap \mathscr{B}\neq \emptyset$.
\end{Lemma}

%

\section{Proof of Theorem~\ref{thm:single far-field}}\label{sect:main
proof}

The entire section is devoted to the proof of
Theorem~\ref{thm:single far-field} by contradiction. Assume that
$\mathbf{D}\neq \widetilde{\mathbf{D}}$ and
\begin{equation}\label{eq:thm 1}
\mathbf{E}_\infty({\hat{x}}; \mathbf{D}, p_0, k_0,
d_0)=\mathbf{E}_\infty({\hat{x}}; \mathbf{\widetilde{D}}, p_0, k_0,
d_0)\qquad \mbox{for\, $\hat{{x}}\in \mathbb{S}^2$}.
\end{equation}
Let $\Omega$ be the unbounded connected component of
$\mathbb{R}^3\backslash (\mathbf{D}\cup
\widetilde{\mathbf{D}})$.\footnote{Since both $\mathbf{D}$ and
$\widetilde{\mathbf{D}}$ are compact sets, we know that $\Omega$ is
unique. Moreover, it is obvious that $\partial \Omega$ forms the
boundary of a polyhedral domain in $\mathbf{G}$.} By Rellich's
theorem (see Theorem~6.9, \cite{ColKre}), we infer from (\ref{eq:thm
1}) that
\begin{equation}\label{eq:n1}
\mathbf{E}({x};\mathbf{D})=\mathbf{E}({x};\widetilde{\mathbf{D}})\qquad
\mbox{for\, ${x}\in\Omega$},
\end{equation}
where $\mathbf{E}({x};\mathbf{D})$ and
$\mathbf{E}({x};\widetilde{\mathbf{D}})$ are, respectively,
abbreviations $\mathbf{E}({x};\mathbf{D},p_0,k_0,d_0)$ and
$\mathbf{E}({x};\widetilde{\mathbf{D}},p_0$ $,k_0,d_0)$. Next,
noting that $\mathbf{D}\neq \widetilde{\mathbf{D}}$, we see that
either $(\mathbb{R}^3\backslash \bar{\Omega})\backslash
\mathbf{D}\neq \emptyset$ or $(\mathbb{R}^3\backslash
\bar{\Omega})\backslash \widetilde{\mathbf{D}}\neq \emptyset$.
Without loss of generality, we assume the former case and let
$D^*:=(\mathbb{R}^3\backslash \bar{\Omega})\backslash \mathbf{D}\neq
\emptyset$. It is easily seen that $D^*\subset
\widetilde{\mathbf{D}}$, so $D^*$ is bounded. Moreover, by choosing
connected component if necessary, we assume that $D^*$ is connected.
Clearly, $D^*$ is a bounded polyhedral domain in
$\mathbf{G}=\mathbb{R}^3\backslash \mathbf{D}$ and
$\mathbf{E}({x};\mathbf{D})$ is defined over ${D}^*$. Noting
$\partial D^*\subset \partial \Omega\cup
\partial \mathbf{D}\subset \partial \mathbf{D}\cup \partial \widetilde{\mathbf{D}}$ and using
(\ref{eq:n1}), we have from the perfect boundary conditions of
$\mathbf{E}({x};\mathbf{D})$ and
$\mathbf{E}({x};\widetilde{\mathbf{D}})$ on $\partial\mathbf{D}$ and
$\partial\widetilde{\mathbf{D}}$ that
\begin{equation}
\nu\times \mathbf{E}({x};\mathbf{D})=0\quad \mbox{on\ \ $\partial
D^*$}.
\end{equation}
In the following, in order to simply notations, we use as those
introduced in Section~2, e.g., we write $\mathbf{E}({x})$ to denote
$\mathbf{E}({x};\mathbf{D})$ etc.. The rest of the proof will be
proceeded into three steps and a brief outline is as follows. In the
first step, we will find a perfect plane $\widetilde{\Pi}_1\in
\mathcal{S}_\mathbf{E}$, and this is the starting point of the
subsequent path argument. In the second step, we would construct
implicitly an exit path, which is a regular curve lying entirely in
the exterior of $\mathbf{D}$ and connecting to infinity. As we
mentioned earlier that the path argument might be broken down with
the presence of some ``unbounded" perfect planes (namely, perfect
planes in $\mathcal{Q}_{\mathbf{E}}$), in order to avoid our
subsequent argument being trapped at such ``unbounded" perfect
planes, the curve is required to have at most one intersection with
$\mathcal{Q}_{\mathbf{E}}$. Fortunately, this can be done by using
Lemma~\ref{lem:topology}. Finally, using the reflection principle in
Theorem~\ref{thm:reflection}, we make continuation of (bounded)
perfect planes along the exit path to find a sequence of perfect
planes. Then a contradiction is constructed by showing that the
continuation must follow the exit path to infinity since we always
step a length larger than a fixed positive constant when making such
continuation, but on the other hand, all the bounded perfect planes
are contained in the convex hull of $\mathbf{D}$ being bounded. In
this final step, we must be carefully treating the possible presence
of ``unbounded" perfect planes and this is the main difference of
the present path argument from those implemented in
\cite{LiuYamZou1} and \cite{LiuYamZou2}.

\vskip 3mm

\noindent \textbf{Step~I: Existence of a bounded perfect plane
$\widetilde{\Pi}_1$ with bounded
$\mathbf{\Lambda}_{\widetilde{\Pi}_1}$}

\vskip 3mm

We first note that $\partial D^*\backslash
\partial \mathbf{D}\neq \emptyset$. Hence, there must be an
open face say $\Sigma_0$ on $\partial D^*$ that can be extended in
$\mathbf{G}$ to form a perfect plane and it is denoted by
$\widetilde{\Pi}_0$. Since $\widetilde{\Pi}_0$ is extended from a
face of the bounded polyhedral domain $D^*$ in $\mathbf{G}$, we
infer from the following Lemma~\ref{lem:intersection} that
$\widetilde{\Pi}_0$ is bounded. Now, if the symmetric set
$\mathbf{\Lambda}_{\widetilde{\Pi}_0}$ corresponding to
$\widetilde{\Pi}_0$ is bounded, then we are done since we can take
$\widetilde{\Pi}_0$ as $\widetilde{\Pi}_1$. So, without loss of
generality, we assume that $\mathbf{\Lambda}_{\widetilde{\Pi}_0}$ is
unbounded. Next, based on $\Sigma_0$, we construct a bounded
polyhedral domain in $\mathbf{G}$ which is symmetric with respect to
$\Pi_0$ but different from $\mathbf{\Lambda}_{\widetilde{\Pi}_0}$.
The construction procedure is similar to that for
$\mathbf{\Lambda}_{\widetilde{\Pi}_0}$, and we nonetheless present
it here for clearness.

Fix an arbitrary point $x^*\in \Sigma_0$ and let
$\mathrm{B}^*:=B_\varepsilon(x^*)$ with $\varepsilon>0$ sufficiently
small such that $\mathrm{B}^*$ is divided by $\Sigma_0$ into two
(open) half balls $\mathrm{B}_*^+$ and $\mathrm{B}_*^-$ satisfying
$\mathrm{B}_*^+\subset D^*$ and $\mathrm{B}_*^-\subset
\mathbf{G}\backslash D^*$. Next, let $\Theta_*^+$ be the connected
component of $\mathscr{R}_{\Pi_0}(\mathbf{G}\backslash
\overline{D^*})\cap D^*$ containing $\mathrm{B}_*^+$ and
$\Theta_*^-$ be the connected component of
$\mathscr{R}_{\Pi_0}D^*\cap (\mathbf{G}\backslash \overline{D^*})$
containing $\mathrm{B}_*^-$. Set $\Theta^*=\Theta_*^+\cup
\Sigma_0\cup \Theta_*^-$. Clearly, $\Theta^*$ is a non-empty bounded
polyhedral domain in $\mathbf{G}$ since $\mathrm{B}^*\subset
\Theta^*\subset D^*\cup \mathscr{R}_{\Pi_0}D^*$. We remark that
$\Theta_0$ is in fact the connected component of $(D^*\cup
\mathscr{R}_{\Pi_0}D^*)\cap \mathbf{\Lambda}_{\widetilde{\Pi}_0}$
containing $\Sigma_0$. By the reflection principle (ii) of
Theorem~\ref{thm:reflection}, $\nu\times \mathbf{E}(x)=0$ on
$\partial \Theta^*$. It is obvious that $\partial \Theta^*\backslash
\mathbf{D}\neq \emptyset$. Let $\Sigma_1\subset \partial
\Theta^*\backslash \mathbf{D}$ be an open face. By analytic
continuation, $\Sigma_1$ is extended in $\mathbf{G}$ to give a
perfect plane $\widetilde{\Pi}_1$. Since $\Theta^*$ is symmetric
with respect to $\Pi_0$, we know
$\Sigma_1\subset\hspace*{-3.6mm}\backslash \ \Pi_0$ and therefore
$\widetilde{\Pi}_1\in\hspace*{-2.6mm}\backslash\,
\mathcal{Q}_{\mathbf{E}}$ by Lemma~\ref{lem:b conplane ub}, i.e.,
$\widetilde{\Pi}_1$ is bounded with bounded
$\mathbf{\Lambda}_{\widetilde{\Pi}_1}$.

\vskip 3mm

\noindent \textbf{Step~II: Construction of the exit path $\gamma$}

\vskip 3mm

Since both $\partial \mathbf{G}$ and $\widetilde{\Pi}_1$ are
bounded, we see that $\mathbf{G}\backslash \widetilde{\Pi}_1$ has a
unique unbounded connected component, which is denoted by
$\mathscr{U}$. It readily has that $\widetilde{\Pi}_1\subset
\partial \mathscr{U}$ and $\mathscr{U}$ contains the exterior of a sufficiently large
ball containing $\mathbf{D}$. Next, we fix an arbitrarily point
$x_1\in \widetilde{\Pi}_1$. Let $\gamma:=\gamma(t) (t\geq 0)$ be a
regular curve such that $\gamma(t_1)=x_1$ with $t_1=0$ and
$\gamma(t) (t>0)$ lies entirely in $\mathscr{U}$. Furthermore,
$\gamma$ connects to infinity, i.e.,
$\lim_{t\rightarrow\infty}|\gamma(t)|=\infty$. The exit path
$\gamma$ constructed in this way might have non-empty intersection
with $\mathcal{Q}_\mathbf{E}$. In this case, we require that
$\gamma(t) (t>0)$ has only one intersection point with
$\mathcal{Q}_\mathbf{E}$. In fact, in case $\gamma(t)\cap
\mathcal{Q}_\mathbf{E}\neq \emptyset$, we would modify the curve
$\gamma$ as follows to satisfy such requirement. Let
$x_T:=\gamma(T)$ be the ``first" intersection point of $\gamma(t)
(t>0)$ and $\mathcal{Q}_\mathbf{E}$; that is,
\[
T=\min\{t>0;\ \gamma(t)\in \mathcal{Q}_\mathbf{E}\}<\infty.
\]
Then, set $\mathscr{V}$ be the connected component of
$\mathbf{G}\backslash \mathcal{Q}_\mathbf{E}$ such that $x_T\in
\partial \mathscr{V}$. Let
$\mathscr{W}:=\mathscr{U}\cap\mathscr{V}$. It can be verified that
$\mathscr{W}$ is an unbounded connected open set such that
$x_T\in\partial\mathscr{W}$. Indeed, the connectedness of
$\mathscr{W}$ is obvious by noting that both $\mathscr{U}$ and
$\mathscr{V}$ are connected. Whereas the unboundedness of
$\mathscr{W}$ is due to the facts that $\mathscr{V}$ is unbounded by
Lemma~\ref{lem:topology} and $\mathscr{U}$ contains the exterior of
a sufficiently large ball containing $\mathbf{D}$ as mentioned
earlier. Next, let $\eta(t)(t\geq T)$ be a regular curve such that
$\eta(T)=x_T$, $\eta(t)(t> T)$ lies entirely in $\mathscr{W}$ and
connects to infinity (i.e.,
$\lim_{t\rightarrow\infty}|\eta(t)|=\infty$). Furthermore, it is
trivially required that $\eta(t)$ has $C^1$-connection with
$\gamma(t)(0\leq t\leq T)$ at $x_T$. Now, set
\[
\tilde{\gamma}(t)=
\begin{cases}
& \gamma(t)\qquad 0\leq t\leq T,\\
& \eta(t)\qquad t>T,
\end{cases}
\]
then $\tilde{\gamma}(t) (t\geq 0)$ satisfies all our requirements of
an exit path.

\vskip 3mm

\noindent \textbf{Step~III: Continuation of bounded perfect planes
along $\gamma$}

\vskip 3mm

Let $d_0=\mathbf{d}(\gamma, \mathbf{D})>0$, which is attainable
since $\mathbf{D}$ is compact, and $r_0=d_0/2$. Clearly,
$\bar{B}_{r_0}(\gamma(t))\subset \mathbf{G}$ for any $t\geq 0$. Let
${\widetilde{x}}_2^+=\gamma(\tilde{t}_2)\in S_{r_0}(x_1)\cap
\gamma$, where $\tilde{t}_2$ is taken to be $\tilde{t}_2=\max \{t>0;
\gamma(t)\in S_{r_0}({x}_1)\}$, and let ${\widetilde{x}}_2^-$ be the
symmetric point of ${{\widetilde{x}}}_2^+$ with respect to $\Pi_1$.
Next, let $\mathbf{G}_1^+$ be the connected component of
$\mathbf{G}\backslash \widetilde{\Pi}_1$ containing
${\widetilde{x}}_2^+$, and $\mathbf{G}_1^-$ be the connected
component of $\mathbf{G}\backslash \widetilde{\Pi}_1$ containing
${\widetilde{x}}_2^-$. Then let $\mathbf{\Lambda}_1^+$ be the
connected component of $\mathbf{G}_1^+\cap
R_{\Pi_1}(\mathbf{G}_1^-)$ containing ${\widetilde{x}}_2^+$ and
$\mathbf{\Lambda}_1^-$ be the connected component of
$\mathbf{G}_1^-\cap R_{\Pi_1}(\mathbf{G}_1^+)$ containing
${\widetilde{x}}_2^-$. Set
$\mathbf{\Lambda}_1=\mathbf{\Lambda}_1^+\cup \widetilde{\Pi}_1\cup
\mathbf{\Lambda}_1^-$. In fact, $\mathbf{\Lambda}_1$ is the
symmetric set $\mathbf{\Lambda}_{\widetilde{\Pi}_1}$ corresponding
to the perfect plane $\widetilde{\Pi}_1$ and we present its
construction again for convenience of the subsequent argument. Since
$\widetilde{\Pi}_1\in \mathcal{S}_\mathbf{E}$, $\mathbf{\Lambda}_1$
is bounded. By Lemma~\ref{lem:fundamental property}, it is easy to
deduce that $\gamma\cap \partial \mathbf{\Lambda}_1 \neq \emptyset$.
We let ${x}_2=\gamma(t_2)$ be the `last' intersection point of
$\gamma$ and $\partial \mathbf{\Lambda}_1$; namely, $t_2=\max\{t>0;
\gamma(t)\in \partial \mathbf{\Lambda}_1\}<\infty$. This then
implies the existence of a perfect plane passing through $x_2$ which
is extended from an open face of $\partial \mathbf{\Lambda}_1$ whose
closure contains $x_2$. We denote the perfect plane by
$\widetilde{\Pi}_2$. Without loss of generality, we further assume
that $x_2$ is the `last' intersection point of $\gamma$ with
$\widetilde{\Pi}_2$. By the following result, we know
$\widetilde{\Pi}_2$ is bounded. We shall prove at the end of this
section:

\begin{Lemma}\label{lem:intersection}
Suppose that $\mathbf{\Lambda}\subset \mathbf{G}$ is a bounded
polyhedral domain such that
\[
\nu\times \mathbf{E}=0\quad \mbox{on \ \ $\partial
\mathbf{\Lambda}$}.
\]
Then every open face lying on $\partial \mathbf{\Lambda}\backslash
\mathbf{D}$ cannot be connectedly extended to an unbounded planar
domain in $\mathbf{G}$.
\end{Lemma}

Now, we still need to distinguish between two cases of
$\widetilde{\Pi}_2\in \breve{\mathcal{Q}}_\mathbf{E}$ and
$\widetilde{\Pi}_2\in \mathcal{S}_\mathbf{E}$. But for the end of a
more general discussion, we next give the induction procedure for
the above reflection argument of finding a different perfect plane
with a known one. Suppose that $\widetilde{\Pi}_n\in
\mathcal{S}_\mathbf{E}$, $n\in \mathbb{N}$ and $x_n:=\gamma(t_n)\in
\gamma\cap \widetilde{\Pi}_n$ is the `last' intersection point
between $\gamma$ and $\widetilde{\Pi}_n$.

Let ${\widetilde{x}}_{n+1}^+=\gamma(\tilde{t}_{n+1})\in
S_{r_0}(x_n)\cap \gamma$, where $\tilde{t}_{n+1}$ is taken to be
$\tilde{t}_{n+1}=\max \{t>0; \gamma(t)\in S_{r_0}({x}_n)\}$, and let
${\widetilde{x}}_{n+1}^-$ be the symmetric point of
${{\widetilde{x}}}_{n+1}^+$ with respect to $\Pi_1$. Next, let
$\mathbf{G}_n^+$ be the connected component of $\mathbf{G}\backslash
\widetilde{\Pi}_n$ containing ${\widetilde{x}}_{n+1}^+$, and
$\mathbf{G}_n^-$ be the connected component of $\mathbf{G}\backslash
\widetilde{\Pi}_n$ containing ${\widetilde{x}}_{n+1}^-$. Then let
$\mathbf{\Lambda}_n^+$ be the connected component of
$\mathbf{G}_n^+\cap R_{\Pi_n}(\mathbf{G}_n^-)$ containing
${\widetilde{x}}_{n+1}^+$ and $\mathbf{\Lambda}_n^-$ be the
connected component of $\mathbf{G}_n^-\cap
R_{\Pi_n}(\mathbf{G}_n^+)$ containing ${\widetilde{x}}_{n+1}^-$. Set
$\mathbf{\Lambda}_n=\mathbf{\Lambda}_n^+\cup \widetilde{\Pi}_n\cup
\mathbf{\Lambda}_n^-$. By our earlier discussion,
$\mathbf{\Lambda}_n=\mathbf{\Lambda}_{\widetilde{\Pi}_n}$, and it is
bounded since $\widetilde{\Pi}_n\in \mathcal{S}_\mathbf{E}$. By
Lemma~\ref{lem:fundamental property}, $\gamma\cap \partial
\mathbf{\Lambda}_n\neq \emptyset$. We let $x_{n+1}:=\gamma(t_{n+1})$
with $t_{n+1}=\max\{t>0; \gamma(t)\in \partial
\mathbf{\Lambda}_n\}<\infty$. This again implies the existence of a
perfect plane $\widetilde{\Pi}_{n+1}$ passing through $x_{n+1}$, and
$\widetilde{\Pi}_{n+1}$ is bounded by Lemma~\ref{lem:intersection}.
Furthermore, we can also assume that $x_{n+1}$ is the `last'
intersection point of $\gamma$ with $\widetilde{\Pi}_{n+1}$. In the
following, we list several important results that have been
achieved:
\begin{enumerate}
\item[]\vspace*{-5mm}
\begin{equation}\label{eq:bounded pf}
\hspace*{-4.2cm}\mbox{(i)}~~x_n,x_{n+1}\in\widehat{\mathscr{P}}_\mathbf{E}:=\{x\in\mathscr{P}_\mathbf{E};\
x\in \widetilde{\Pi}\ \mbox{with $\widetilde{\Pi}\in
\mathcal{S}_\mathbf{E}\cup \breve{\mathcal{Q}}_\mathbf{E}$} \};
\end{equation}

\item[(ii)]$\widetilde{\Pi}_{n+1}$ is different from
$\widetilde{\Pi}_n$, since $t_{n}$ and $t_{n+1}$ with $t_{n+1}>t_n$
are respectively the `last' intersection points between $\gamma$ and
$\widetilde{\Pi}_n$ and $\widetilde{\Pi}_{n+1}$;

\item[(iii)] Both $\widetilde{\Pi}_n$ and $\widetilde{\Pi}_{n+1}$
are bounded;

\item[(iv)] Since $B_{r_0}(x_n)\subset \mathbf{\Lambda}_n$, the
length of $\gamma(t)$ from $t_n$ to $t_{n+1}$ is not less than
$r_0$, i.e.,
\begin{equation}
|\gamma(t_n\leq t\leq t_{n+1})|\geq |\gamma(t_n\leq t\leq
\tilde{t}_{n+1})|\geq r_0.
\end{equation}
\end{enumerate}


If $\widetilde{\Pi}_{n+1}\in \mathcal{S}_\mathbf{E}$, by repeating
the above reflection argument, we can find another bounded perfect
plane $\widetilde{\Pi}_{n+2}$, and also $x_{n+2}:=\gamma(t_{n+2})$,
the `last' intersection point between $\gamma$ and
$\widetilde{\Pi}_{n+2}$, such that
\[
|\gamma(t_{n+1}\leq t\leq t_{n+2})|\geq r_0.
\]

In case $\widetilde{\Pi}_{n+1}\in \breve{\mathcal{Q}}_{\mathbf{E}}$,
we can no longer guarantee that $\gamma\cap \partial
\mathbf{\Lambda}_{n+1}\neq \emptyset$ since
$\mathbf{\Lambda}_{n+1}=\mathbf{\Lambda}_{\widetilde{\Pi}_{n+1}}$ is
unbounded. Let $\mathrm{B}_0:=B_{\epsilon_0}(x_{n+1})$ with
$\epsilon_0>0$ sufficiently small such that $\mathrm{B}_0\subset
B_{r_0}(x_{n+1})$ and one of the half ball of $B_0$ divided by
$\widetilde{\Pi}_{n+1}$ is contained entirely in
$\mathbf{\Lambda}_{n}$.\footnote{Here, we recall that
$\widetilde{\Pi}_{n+1}$ is extended from an open face of
$\mathbf{\Lambda}_{n}$.} Then, let $\mathbf{\Lambda}_{n+1}^*$ be the
connected component of $(\mathbf{\Lambda}_n\cup
\mathscr{R}_{\Pi_{n+1}}\mathbf{\Lambda}_n)\cap
\mathbf{\Lambda}_{n+1}$ containing $\mathrm{B}_0$.\footnote{This is
similar to the construction of $\Theta^*$ from $D^*$ in Step I of
the present proof.} Since $\mathbf{\Lambda}_n$ is bounded, we know
$\mathbf{\Lambda}_{n+1}^*$ is bounded. Moreover, by the reflection
principle in (ii) of Theorem~\ref{thm:reflection}, $\nu\times
\mathbf{E}(x)=0$ on $\partial \mathbf{\Lambda}_{n+1}^*$. Now, by
Lemma~\ref{lem:fundamental property}, it is verified directly that
$\gamma\cap \mathbf{\Lambda}_{n+1}^*\neq \emptyset$. Also, we let
$x_{n+2}:=\gamma(t_{n+2})$ be the `last' intersection point between
$\gamma$ and $\partial\mathbf{\Lambda}_{n+2}^*$. By analytic
continuation, this implies the existence of a perfect plane
$\widetilde{\Pi}_{n+2}$ passing through $x_{n+2}$, which must be
bounded by Lemma~\ref{lem:intersection}. More importantly, noting
that $\widetilde{\Pi}_{n+2}$ is not conplane to
$\widetilde{\Pi}_{n+1}$, we know by Lemma~\ref{lem:b conplane ub}
that $\widetilde{\Pi}_{n+2}\in \mathcal{S}_\mathbf{E}$. As what has
been frequently done before, we can further assume that $x_{n+2}$ is
the `last' intersection point of $\gamma$ with
$\widetilde{\Pi}_{n+2}$. Finally, it is easy to show
\[
|\gamma(t_{n+1}\leq t\leq t_{n+2})|\geq \epsilon_0.
\]

By induction and also by noting that $\gamma(t) (t>0)$ has at most
one intersection point with $\mathcal{Q}_\mathbf{E}$, we have
constructed a sequence of different perfect planes
$\widetilde{\Pi}_n,\ n=1,2,3,\ldots$, all belonging to
$\mathcal{S}_\mathbf{E}$ except possibly only one belonging to
$\breve{\mathcal{Q}}_\mathbf{E}$. Moreover, there is a strictly
increasing sequence $\{t_n\}_{n=1}^\infty$ together with a sequence
of points $x_n=\gamma(t_n)\in \gamma\cap \widetilde{\Pi}_n,\
n=1,2,3,\ldots$, such that
\begin{equation}\label{eq:inequality}
|\gamma(t_n\leq t\leq t_{n+1})|\geq r_0\qquad \mbox{when\ $n>n_0$},
\end{equation}
where $n_0$ is the index such that $\widetilde{\Pi}_{n_0}\in
\breve{\mathcal{Q}}_\mathbf{E}$ and it might be 0.

Now, we can conclude our proof of Theorem~\ref{thm:single far-field}
by a contradiction as follows. Since $\gamma(t_n)\in
\widehat{\mathscr{P}}_\mathbf{E}\subset \overline{ch (\mathbf{D})}$
being bounded and $\lim_{t\rightarrow\infty}|\gamma(t)|=\infty$, we
know there must exist some $T_0<\infty$ such that
$\lim_{n\rightarrow\infty}t_n=T_0$. Then,
\begin{equation}\label{eq:contradiction}
\lim_{n\rightarrow\infty}|\gamma(t_n\leq t\leq
t_{n+1})|=\lim_{n\rightarrow\infty}\int_{t_n}^{t_{n+1}}
|\gamma'(t)|\ dt=0.
\end{equation}
A contradiction to (\ref{eq:inequality}). \hfill $\Box$

\begin{proof}[Proof of Lemma~\ref{lem:intersection}]
Assume contrarily that there is an open face $\Gamma_0$ on
$\partial\mathbf{\Lambda}\backslash \mathbf{D}$ which can be
connectedly extended in $\mathbf{G}$ to give an unbounded planar
domain. By analytic continuation, this gives an unbounded perfect
plane $\widetilde{\Pi}_0$. Since $\mathbf{\Lambda}$ is a bounded
polyhedron in $\mathbf{G}$, $\widetilde{\Pi}_0$ must be separated
from $\mathbf{\Lambda}$ at some of its edge. Hence, there is another
open face $\Gamma_1$ on $\partial\mathbf{\Lambda}\backslash
\mathbf{D}$, such that $\Gamma_0$ and $\Gamma_1$ have a common edge
in $\mathbf{G}$. Again by analytic continuation, we have a perfect
plane $\widetilde{\Pi}_1$ from the connected extension of
$\Gamma_1$ in $\mathbf{G}$. Noting $\widetilde{\Pi}_0\in
\mathcal{Q}_\mathbf{E}$, we see $\widetilde{\Pi}_1\in
\mathcal{S}_\mathbf{E}$ by Lemma~\ref{lem:b conplane ub}. Next, the
argument follows a similar manner as that of Step III in the proof
of Theorem~\ref{thm:single far-field}.

Fix an arbitrary point $x_1\in \Gamma_0\cap\Gamma_1$. Let
$\gamma:=\gamma(t) (t\geq 0)$ be a regular curve such that
$\gamma(t_1)={x}_1$ with $t_1=0$ and $\gamma(t)(t>0)$ lies entirely
in the unbounded connected component of $\widetilde{\Pi}_0\backslash
\mathbf{\Lambda}$ and $\lim_{t\rightarrow\infty}|\gamma(t)|=\infty$.
Set $\tau_0=\mathbf{d}(\gamma, \mathbf{D})>0$. From our earlier
discussion in Step III of the proof of Theorem~\ref{thm:single
far-field}, we know $\gamma\cap \partial
\mathbf{\Lambda}_{\widetilde{\Pi}_1}\neq \emptyset$. Furthermore,
letting $x_2:=\gamma(t_2)$ be the `last' intersection point of
$\gamma$ with $\partial \mathbf{\Lambda}_{\widetilde{\Pi}_1}$, there
is another perfect plane $\widetilde{\Pi}_2$ extended from an open
face on $\partial \mathbf{\Lambda}_{\widetilde{\Pi}_1}$ such that
$\widetilde{\Pi}_2$ passes through $x_2$ and
\[
|\gamma(t_1\leq t\leq t_2)|\geq \tau_0.
\]
A crucial observation is that $\gamma\subset \widetilde{\Pi}_0$, we
can without loss of generality assume that $\widetilde{\Pi}_2$ is
non-parallel to $\widetilde{\Pi}_0$, therefore $\widetilde{\Pi}_2\in
\mathcal{S}_\mathbf{E}$ by Lemma~\ref{lem:b conplane ub}. By
repeating the above procedure, we can construct countably many
different perfect planes $\widetilde{\Pi}_n\in
\mathcal{S}_\mathbf{E},\ n=1,2,3,\ldots,$, together with a sequence
of points $x_n:=\gamma(t_n)\in \gamma\cap\widetilde{\Pi}_n$
satisfying
\[
|\gamma(t_n\leq t\leq t_{n+1})|\geq \tau_0.
\]
Finally, a similar contradiction is established as that in
(\ref{eq:contradiction}), thus completing the proof.
\end{proof}

\section{Concluding Remarks}

In this paper, we have established a global uniqueness for the
formally determined inverse electromagnetic obstacle scattering.
That is, the far-field pattern $\mathbf{E}_\infty({\hat{x}};
\mathbf{D}, p_0,$ $ k_0, d_0)$ for fixed $p_0\in \mathbb{R}^3$,
$k_0>0$, $d_0\in\mathbb{S}^2$ and all ${\hat{x}}\in \mathbb{S}^2$,
uniquely determine a general polyhedral scatterer $\mathbf{D}$. As
mentioned in the introduction, some uniqueness results on the unique
determination of general polyhedral obstacles have been established,
but all with the far-field patterns corresponding to two different
incident waves.

In \cite{LiuYamZou1}, the underlying obstacle admits the
simultaneous presence of finitely many cracks, where a \emph{crack}
is defined to be the closure of some bounded open subset of a plane
in $\mathbb{R}^3$. That is, in addition to finitely many solid
polyhedra, the polyhedral obstacle $\mathbf{D}$ in \cite{LiuYamZou1}
may also contains finitely many cracks. In the case with the
additional presence of a crack to the polyhedral obstacle
$\mathbf{D}$ considered in Theorem~\ref{thm:single far-field}, one
verifies straightforwardly that the argument in Step~I of its proof
might not hold any longer. In fact, one may not be able to find a
bounded polyhedral domain $D^*$ in $\mathbf{G}$, and it might be a
sole crack instead. In turn, one may not be able to construct the
bounded polyhedral domain $\Theta^*$, which is essential to find the
starting perfect plane $\widetilde{\Pi}_1\in \mathcal{S}_\mathbf{E}$
for the subsequent path argument.

Since knowing $\mathbf{E}_\infty({\hat{x}}; \mathbf{D})$ and
$\mathbf{H}_\infty({\hat{x}};\mathbf{D})$ are equivalent, one can
see that Theorem~\ref{thm:single far-field} is still valid with the
polyhedral obstacle $\mathbf{D}$ associated with the following
perfect boundary condition corresponding to $\mathbf{H}$
\begin{equation}\label{eq:perfect h}
\nu\times \mathbf{H}=0\quad \mbox{on\ \ $\partial \mathbf{G}$}.
\end{equation}
In \cite{LiuYamZou2}, a more general situation is considered that we
need not to know the \emph{a priori} physical properties of the
underlying obstacle. That is, the underlying obstacle $\mathbf{D}$
may be either associated with boundary condition (\ref{eq:perfect
bc}), or (\ref{eq:perfect h}), or even with mixed type of
(\ref{eq:perfect bc}) and (\ref{eq:perfect h}). In such setting, we
need to consider perfect planes corresponding to both the electric
field $\mathbf{E}$ and magnetic $\mathbf{H}$ (see
\cite{LiuYamZou2}). By using a single incident wave, one can show
that Lemma~\ref{lem:conplane} may not hold any longer. In fact, by
direct calculations, two non-parallel unbounded perfect planes, one
corresponding to $\mathbf{E}$ and the other corresponding to
$\mathbf{H}$, may not give a contradiction as that in
Lemma~\ref{lem:nparallel}. Consequently, Theorem~\ref{thm:single
far-field} might not be valid with the underlying polyhedral
obstacle $\mathbf{D}$ associated with mixed boundary conditions.

\section*{Acknowledgement}

The author would like to acknowledge the useful discussion with
Prof. Jun Zou of the Chinese University of Hong Kong and Prof.
Elschner Johannes of the Weierstrass Institute for Applied Analysis
and Stochastics, which is of great help to the current study.

\end{document}